\documentclass[letterpaper, english, 10pt, reqno]{amsart}

\usepackage{amssymb}

\usepackage[english]{babel}
\usepackage[utf8]{inputenc}
\usepackage{booktabs,multirow}
\usepackage{mathtools}
\usepackage{diagbox}

\newtheorem{theorem}{Theorem}
\newtheorem{lemma}[theorem]{Lemma}

\theoremstyle{remark}
\newtheorem{remark}[theorem]{Remark}

\numberwithin{table}{section}
\numberwithin{figure}{section}
\numberwithin{equation}{section}
\numberwithin{theorem}{section}

\renewcommand{\bar}[1]{{\overline{#1}}}
\renewcommand{\hat}[1]{{\widehat{#1}}}
\renewcommand{\tilde}[1]{{\widetilde{#1}}}
\newcommand{\orthProj}{\Pi^{\ell,m}_{\H|h}}

\newcommand{\avg}[1]{{\{\![ #1 ]\!\}}}
\newcommand{\jump}[1]{{[\![ #1 ]\!]}}
\newcommand{\upw}[1]{{\left( #1 \right)_\uparrow}}

\newcommand{\IR}{\mathbb R}

\newcommand{\IN}{\mathbb N}

\newcommand{\IP}{\mathbb P}

\newcommand{\elem}{\mathcal K}
\newcommand{\T}{\mathcal T}
\newcommand{\F}{\mathcal F}
\renewcommand{\H}{\mathcal H}

\renewcommand{\vec}{\mathbf}
\newcommand{\gvec}{\boldsymbol}
\newcommand{\bpi}{\gvec \pi}
\newcommand{\Nu}{\boldsymbol \nu}
\newcommand{\dx}{\, d\vec x}
\newcommand{\ds}{\, d\sigma}

\usepackage[colorlinks = true, linkcolor = blue, citecolor = blue, urlcolor = blue]{hyperref}

\allowdisplaybreaks

\begin{document}

\title[A subcell-enriched Galerkin method] {A subcell-enriched Galerkin method for advection problems} 
\thanks{This work is supported by the Deutsche Forschungsgemeinschaft (DFG, German Research Foundation) under Germany's Excellence Strategy EXC 2181/1 - 390900948 (the Heidelberg STRUCTURES Excellence Cluster).}
\thanks{The authors thank Prof. D.~Kuzmin (TU Dortmund University) for the active participation in discussions regarding this new method and his valuable input at various stages of development of this manuscript.}

\author{Andreas Rupp}
\address{Interdisciplinary Center for Scientific Computing (IWR), Heidelberg University, Mathematikon, Im Neuenheimer Feld 205, 69120 Heidelberg, Germany}
\email{andreas.rupp@fau.de, andreas.rupp@uni-heidelberg.de}

\author{Moritz Hauck}
\address{Department of Mathematics, Friedrich-Alexander University Erlangen-N\"urnberg, Cauerstra\ss e 11, 91058 Erlangen, Germany}
\email{moritz.hauck@fau.de}

\author{Vadym Aizinger}
\address{Chair of Scientific Computing, University of Bayreuth, Universit\"atsstra\ss e 30, 95447 Bayreuth, Germany}
\email{vadym.aizinger@uni-bayreuth.de}


\date{\today}


\begin{abstract} 
 In this work, we introduce a generalization of the enriched Galerkin (EG) method. The key feature of our scheme is an~adaptive two-mesh approach that, in addition to the standard enrichment of a~conforming finite element discretization via discontinuous degrees of freedom, allows to subdivide selected (e.g. troubled) mesh cells in a~non-conforming fashion and to use further discontinuous enrichment on this finer submesh. We prove stability and sharp \textit{a priori} error estimates for a~linear advection equation by using a~specially tailored projection and conducting some parts of a~standard convergence analysis for both meshes. By allowing an arbitrary degree of enrichment on both, the coarse and the fine mesh (also including the case of no enrichment), our analysis technique is very general in the sense that our results cover the range from the standard continuous finite element method to the standard discontinuous Galerkin (DG) method with (or without) local subcell enrichment. Numerical experiments confirm our analytical results and indicate good robustness of the proposed method.
\end{abstract}
\subjclass[2010]{65M60, 65N30}
\keywords{enriched Galerkin method, arbitrary order finite elements, subcell enrichment, advection equation, hyperbolic problem}
\maketitle
%
%
%
\section{Introduction}
The main idea of the enriched Galerkin (EG) method is to extend the approximation space of the continuous finite elements by including some element-local discontinuous functions and to utilize a~solution procedure similar to that of the discontinuous Galerkin (DG) method (Riemann solvers, edge fluxes, \ldots). The latter feature makes the EG schemes fundamentally different from the XFEM methods that frequently also rely on local approximation space enrichments. The resulting discretization is locally conservative and robust but, in multidimensions, has substantially fewer degrees of freedom than a~DG method of the same order.

In \cite{RuppL2020}, the EG methods were re-cast as a~generalization of the classical finite elements, i.e.~continuous Galerkin (CG) methods by considering the EG space as a~combination of arbitrary continuous and discontinuous Galerkin (DG) test and trial spaces. However, the original EG scheme proposed in \cite{Becker2003} for the advection equation was a~combination of lowest order finite elements and finite volumes discretized using the DG framework. This methodology was further developed and investigated by Wheeler, Lee, and coworkers, who also considered higher order enriched CG methods and a~wider range of applications~\cite{lee2016locally,sun2009locally,LEE2019,lee2018phase,lee2017adaptive,lee2018enriched,choo2018enriched,Kadeethum2019,kadeethum2019comparison}. The analysis of EG method in~\cite{RuppL2020} used a~special EG-type projection and was limited to elliptic and parabolic problems. Nonetheless, it paved the way to the analysis for hyperbolic equations in this work.

Similarly to CG approximations, EG methods for hyperbolic equations may develop spurious oscillations. Kuzmin et al.~\cite{KuzminHR2020} proposed several algebraic flux correction schemes to ensure the validity of local maximum principles. Limiting techniques of this kind have also been successfully applied to CG~\cite{Kuzmin2012,Lohmann2019} and DG~\cite{FrankRK2020} discretizations. The use of localized subcell limiters was found to be essential in extensions to high-order Bernstein finite elements \cite{HajdukEtAl2020a,HajdukEtAl2020b,KuzminMQL2020,Lohmann2019}. An $hp$-adaptive approach to subcell limiting was introduced in \cite{KuzminEtAl2019}. Using continuous blending functions, a high-order finite element approximation on a large macrocell was combined with a bound-preserving piecewise (multi-)linear subcell approximation.

Another well-known class of methods relying on subcell limiting to suppress spurious oscillations has been introduced in~\cite{Dumbser2014} and generalized to unstructured meshes in~\cite{Dumbser2016}. These techniques are based on the ADER-DG schemes proposed in~\cite{Dumbser2008} and possess a~very attractive capability to detect high- and low-regularity solution behavior. The underlying {\it a posteriori} limiting strategy was inspired by the Multi-dimensional Optimal Order Detection (MOOD) approach originally developed for finite volumes. In the context of the ADER-DG methods, physical and numerical admissibility conditions are enforced by, first, advancing the solution in time using a~high-order DG method on the coarse mesh, and, for troubled cells, repeating the last time step locally via a~low-order DG (i.e. finite volume) method on the submesh.

Our subcell EG method has the potential to further customize the local approximation space by supporting the whole range of local polynomial orders on both, the coarse and the fine (subcell) mesh. This feature of our approach makes it possible to combine popular $p$- and $hp$-adaptivity techniques with the two-mesh approach , while exploiting its intrinsic ability to assess the local solution regularity.



The main purpose of this work is to present a~stability and \textit{a~priori} error analysis for the subcell-enriched EG method for the linear advection equation and to demonstrate the performance of the new scheme using some test problems. As in~\cite{RuppL2020} this analysis is conducted in a unified framework that covers the CG, DG, and EG (with and without subcell enrichment) discretizations. The implementation of the new numerical scheme was carried out in our FESTUNG\footnote{https://github.com/FESTUNG} framework~\cite{FrankRAK2015,ReuterAWFK2016,JaustRASK2018,ReuterHRFAK2019,ReuterRAFK2020} based on our EG scheme for the shallow-water equations~\cite{HauckAFHR2020}.
\subsection{Model problem}
We consider a non-stationary advection equation on a~bounded Lipschitz domain $\Omega \subset \IR^d$ (with $d \le 3$). The precise formulation of the linear hyperbolic problem to be solved is as follows:
\begin{equation}\label{EQ:advection}
 \partial_t u + \nabla \cdot (\vec a(t, \vec{x})\, u) = f(t, \vec{x}) \qquad \text{ in } (0,T) \times \Omega,
\end{equation}
for a~given velocity field $\vec a \in L^\infty(0,T;W^{1,\infty}(\Omega))$ and a~right-hand side function $f \in L^2((0,T) \times \Omega)$. Additionally, initial data $u_0 \in C(\bar \Omega)$ is prescribed, and we denote by $\Nu_\Omega$ the outward unit normal to $\partial \Omega$. Furthermore, we assume that the inflow boundary
\begin{equation*}
 \Gamma_- ~\coloneqq~ \{ \vec{x} \in \partial \Omega : \vec a(t, \vec{x}) \cdot \Nu_\Omega < 0\}
\end{equation*}
is independent of time and disjointly subdivided into Dirichlet $\Gamma_\textup D$ and flux $\Gamma_\textup F$ boundaries (this subdivision is also assumed to be independent of time), i.e.
\begin{equation*}
 u = u_\textup D \text{ on } \Gamma_\textup D \quad \text{ and } \quad |\vec a \cdot \Nu_\Omega| u = g_\textup F \text{ on } \Gamma_\textup F, \quad u_\textup D \in L^2(0,T;H^{1/2}(\Gamma_\textup D)), \quad g_\textup F \in L^2(0,T;H^{1/2}(\Gamma_\textup F)).
\end{equation*}
%
For the sake of simplicity, we assume that there exists $\delta > 0$ such that $\Gamma_\text F \subset \{ \vec{x} \in \Gamma_- \colon \vec a(t, \vec{x}) \cdot \Nu_\Omega \le -\delta \}$.
\subsection{Structure of the manuscript}
The remainder of this manuscript is structured as follows: In Section~\ref{sec:EG}, we introduce the enriched Galerkin method with local subcell enrichment for advection equations. Section~\ref{SEC:stability} investigates the energy stability of the new scheme, while its \textit{a~priori} convergence is proved in Section~\ref{SEC:error} and verified numerically in Section~\ref{SEC:numerics}. A short conclusions section wraps up the article.
\section{The enriched Galerkin finite element method}
\label{sec:EG}
\subsection{Basic definitions and notations}\label{SEC:basic_definitions}
%
In the following, $(\T_\H)_{\H \in \mathcal I \subset \IR^+}$ denotes a \emph{successively refined} family of $\T_\H \coloneqq \T_\H (\Omega) \coloneqq \{ \elem_i : i = 1, \ldots, N_\textup{el}\}$  ($N_\textup{el}> 0$ is the number of elements) of \emph{$d$-dimensional non-overlapping partitions} of $\Omega$ (see \cite[Def. 1.12]{PietroErn}) that is assumed to be \emph{regular} (in the sense of \cite[Def. 1.38]{PietroErn}) and \emph{geometrically conformal} (in the sense of \cite[Def. 1.55]{ErnGuermond}). For the sake of simplicity, we assume that $\T_\H$ consists of simplices and/or quadrilaterals/hexahedrons. 

Furthermore, $\T_{\H|h}$ denotes a mesh $\T_\H$ of which some elements have been refined (Fig.~\ref{fig:H-mesh} (middle)). The mesh $\T_{\H|h}$ can be geometrically non-conformal. By construction, $\T_{\H|h}$ can be embedded into a regular and conformal mesh $\T_h$ (Fig.~\ref{fig:H-mesh} (right)), which contains the elements added during the refinement process. Hence, we can write $\T_{\H|h}$ as disjoint union
\begin{equation*}
 \T_{\H|h} = \mathcal S_\H \uplus \mathcal S_h \qquad \text{ with } \qquad \mathcal S_\H \subset \T_\H, \; \mathcal S_h \subset \bigcup_{h \le \mathfrak h < \H} \T_\mathfrak h
\end{equation*}
denoting the subsets of unrefined and refined elements, respectively.
Writing $\F(\T_\H)$ for the set of \emph{faces} we define the skeleton of $\T_{\H|h}$ as 
\begin{equation*}
\Sigma \coloneqq \bigcup_{F \in \F(\T_{\H|h})} F.
\end{equation*}

We write $h_\elem$ for the diameter of $\elem$;
furthermore, parameter $\mathcal H$ refers to the maximum diameter of an element of a mesh, i.e., $\mathcal H = \max\{h_\elem : \elem \in \T_\H \}$. If $\Nu$ without an index is evaluated on a face, a unit normal with respect to the face is arbitrarily chosen.

\begin{figure}[ht]
 \includegraphics[draft = false, width=0.32\linewidth]{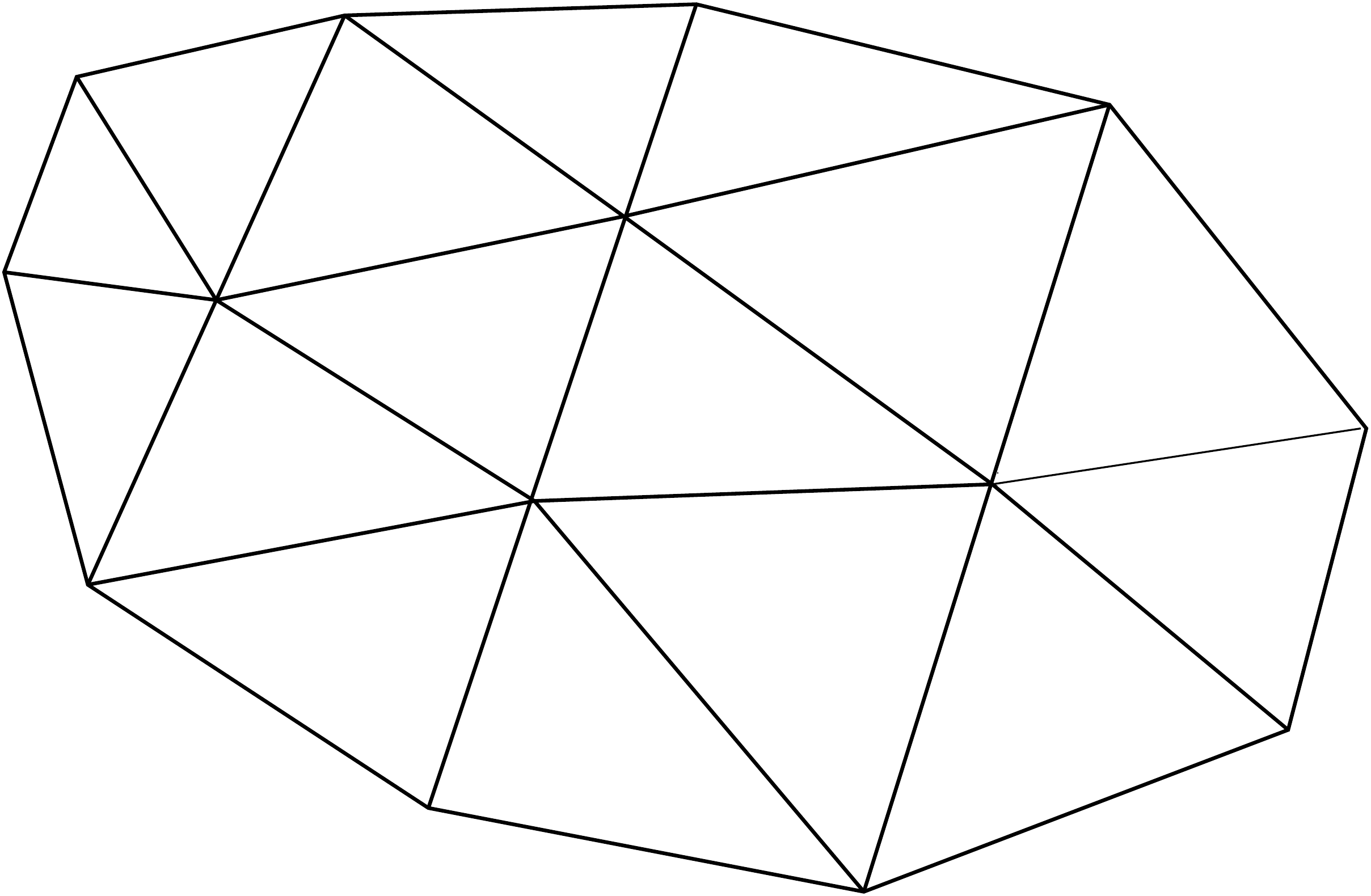}
 \includegraphics[draft = false, width=0.32\linewidth]{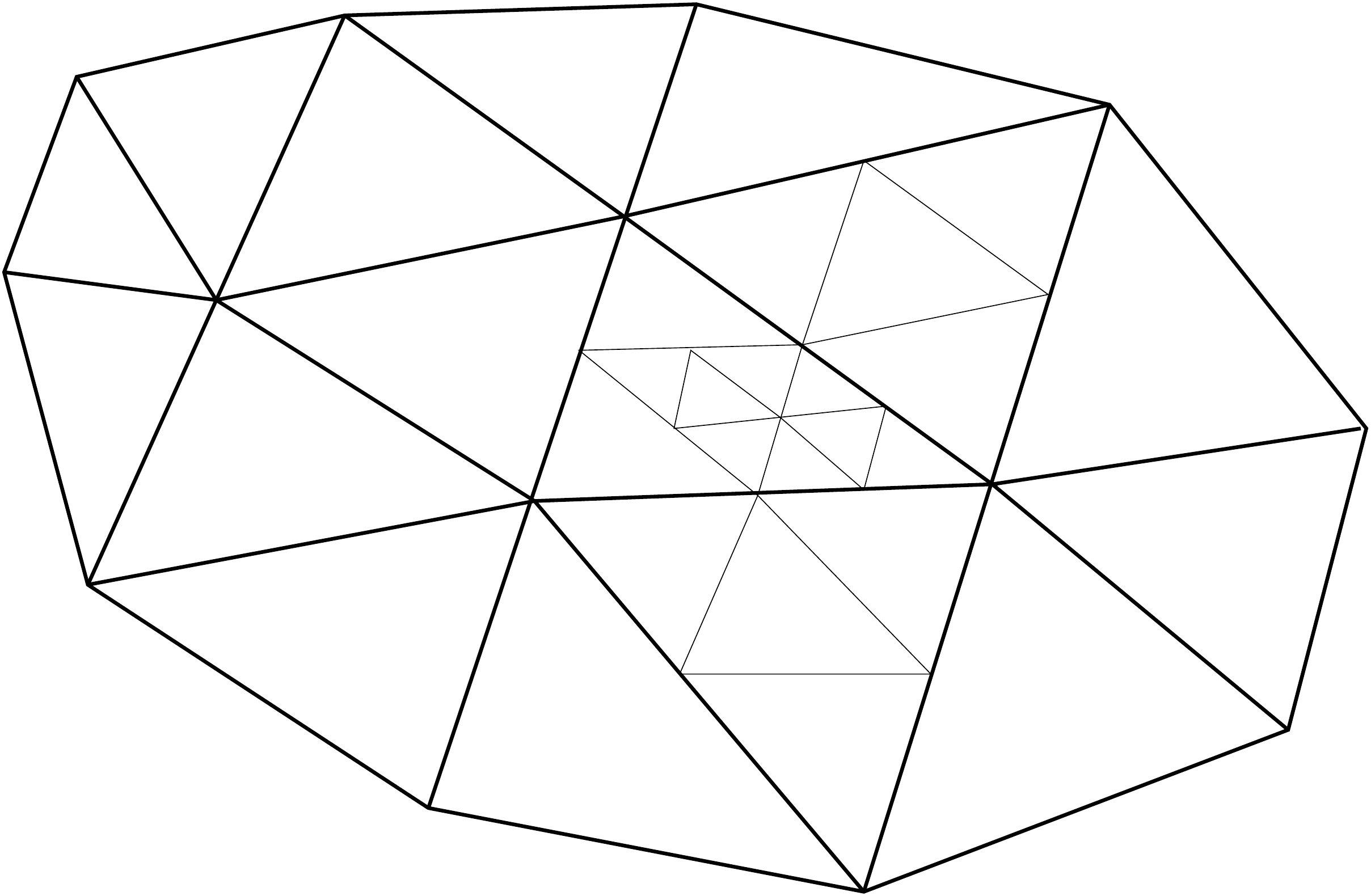}
 \includegraphics[draft = false, width=0.32\linewidth]{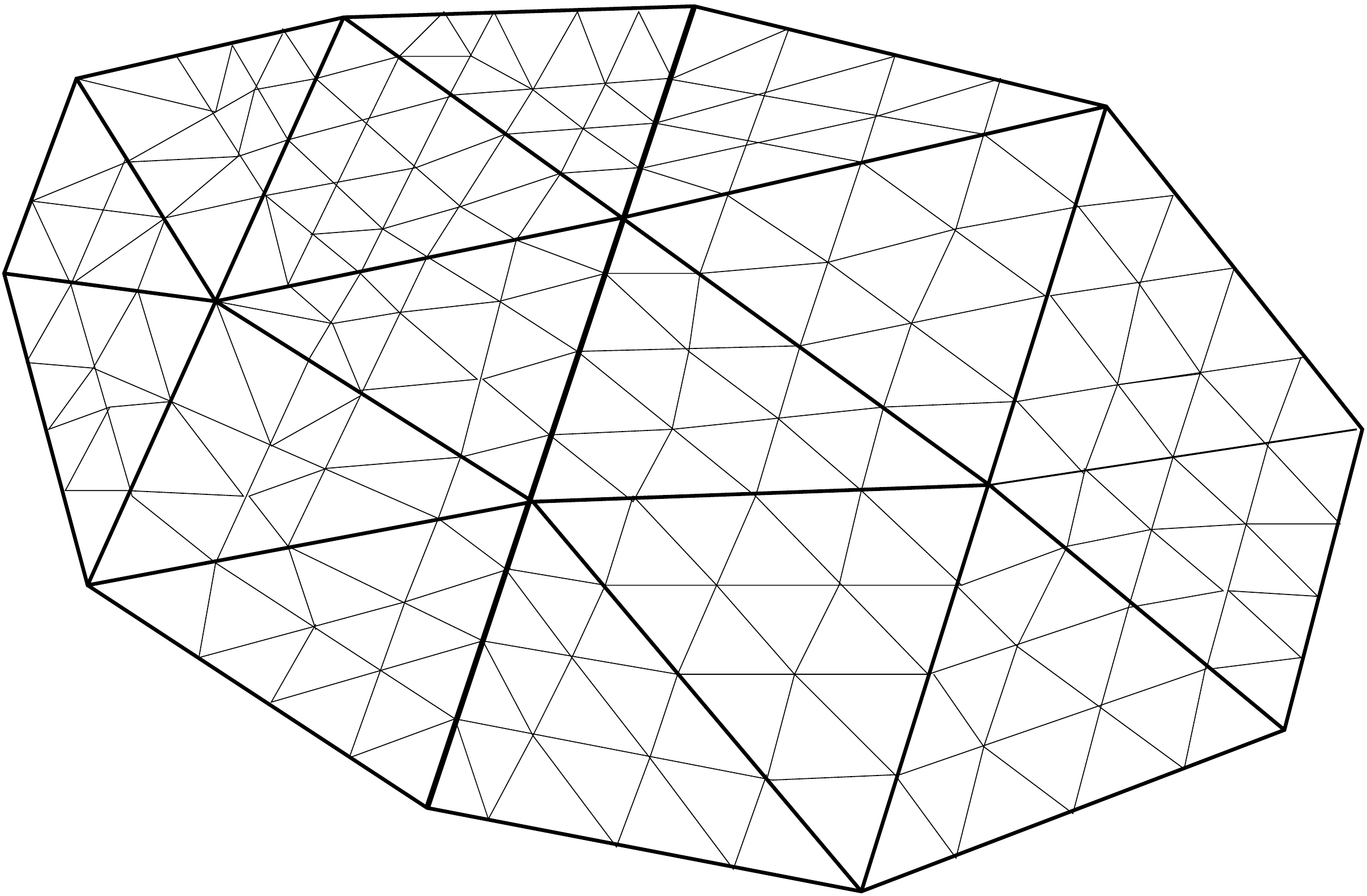}
 \caption{Schematics of $\T_{\H}$ (left), $\T_{\H|h}$ (middle), and $\T_{h}$ (right) meshes.}
\label{fig:H-mesh}
\end{figure}

The double mesh sequence $(\T_{\H|h})^{\H \in \mathcal I}_{h \le \H, h \in \mathcal I}$ is called \emph{weakly quasi-uniform} if there exists a constant $\rho > 0$ such that for all $\H \in \mathcal I$, all $h < \H$, $h \in \mathcal I$, and all $\hat \elem \in \T_\H \setminus \mathcal S_\H$, we have
\begin{align*}
 h_{\H|\hat \elem} ~:=~  \max \{ h_\elem : \elem \in \mathcal S_h \text{ with } \elem \subset \hat \elem \}
 ~\le~ \rho \min \{ h_\elem : \elem \in \mathcal S_h \text{ with } \elem \subset \hat \elem \}.
\end{align*}
To simplify notation we set $h_{\H|\hat \elem} = h_{\hat \elem}$ for $\hat \elem \in \mathcal S_\H$.

The \emph{test} and \emph{trial spaces} for our EG method utilize the \emph{broken polynomial spaces of order $m$} on some mesh $\mathcal M \in \{\T_\H, \T_{\H|h}\}$. They are denoted by $\IP_m(\mathcal M)$ and consist of element-wise polynomials of degree at most $m$ (simplices) or tensor-product polynomials of degree at most $m$ in each spatial coordinate (quadrilaterals/hexahedrons) without any continuity constraints. Thus,
\begin{equation*}
 V^k_{\ell,m} \coloneqq \left( \IP_k(\T_\H) \cap C(\Omega) \right) + \IP_\ell (\T_\H) + \IP_m(\T_{\H|h})
\end{equation*}
for $-1 \le m \le \ell \le k$, $k > 0$. Here, $\IP_{-1}(\mathcal M) = \{0\}$, and one can observe that $\IP_k(\T_\H) \cap C(\Omega)$ is the standard continuous finite element space. 
Obviously $\IP_m(\T_\H) \subset V^k_{\ell,m} \subset \IP_k(\T_h)$. 

In this work, we utilize several types of projection/interpolation operators denoted as follows:
\begin{itemize}
\item $\Pi^r_\H$ and $\Pi^{\ell,m}_{\H|h}$ are the $L^2$-projections into the spaces $\IP_r(\T_\H)$ and $\IP_\ell(\T_\H) + \IP_m(\T_{\H|h})$, respectively.
\item $I^k_\H$ is the standard interpolation operator for finite element space $\IP_k(\T_\H) \cap C(\Omega)$.
\item $\bpi$ is the mapping used to project the initial data into $V^k_{\ell,m}$ proposed in~\cite{RuppL2020} and given by
\begin{equation}
\label{bpi}
\bpi \colon L^2(\Omega) \cap C(\Omega) \to V^k_{\ell,m}, \quad \bpi u \coloneqq I^k_\H u + \orthProj (u - I^k_\H u).
\end{equation}
\end{itemize}
%
\subsection{Semi-discrete formulation}
The semi-discrete EG formulation of the problem can be constructed by using the standard DG bilinear and linear forms for the advection equation on $\T_{\H|h}$. The bilinear form uses the notion of averages $\avg{\cdot}$ and jumps $\jump{\cdot}$, which for $\F \ni F \subset \partial \elem^- \cap \partial \elem^+$ with $\elem^- \neq \elem^+$ are defined as
\begin{equation*}
\avg{g} \coloneqq \frac 12 (g|_{\elem^-} + g|_{\elem^+}), \qquad
\jump{g} \coloneqq g|_{\elem^-} \Nu_{\elem^-} + g|_{\elem^+} \Nu_{\elem^+},
\end{equation*}
%
for a scalar $g$ that is element-wise smooth enough to have traces. On $\partial \Omega$, this definition is modified as follows:
\begin{equation*}
\avg{g} \coloneqq g, \qquad 
\jump g \coloneqq g \Nu_\Omega, 
\end{equation*}
Hence, the jump turns a scalar into a~vector. Also note the following property of jumps used in our analysis
\begin{equation*}
 \jump{g^2} = 2 \avg{g} \jump{g}.
\end{equation*}
Given a~velocity field $\vec a$, we define the upwind value of $g$ as
\begin{equation*}
 \upw{g} \vec a \coloneqq \avg{g} \,\vec a + \frac{\mbox{sign}(\vec a \cdot \Nu)}{2} \jump{g} \cdot \Nu \, \vec a,
\end{equation*}
where $\operatorname{sign}(\cdot)$ is the standard signum function.

Using this notation, we can formulate our semi-discrete problem
\begin{equation*}
a (U, \varphi) = b(\varphi)
\end{equation*}
with trial function $U$ and test function $\varphi$ from $V^k_{\ell,m}$ for almost every $t \in (0,T)$ and $U(0) = \bpi u_0$, where
\begin{align*}
 a (U, \varphi) \coloneqq & \int_\Omega \partial_t U \varphi \dx 
  - \sum_{\elem \in \T_{\H|h}} \int_\elem U \vec a \cdot \nabla \varphi \dx +  \int_{\Sigma \setminus \Gamma_-} \!\!\!\!\!\! \upw{U} \vec a \cdot\jump{\varphi} \ds,\\
 b(\varphi) \coloneqq & \int_\Omega f \varphi \dx + \int_{\Gamma_\textup D} u_\textup D |\vec a \cdot \Nu_\Omega| \varphi \ds + \int_{\Gamma_\textup F} g_\textup F \varphi \ds.
\end{align*}
%

Note that $a(\cdot, \cdot)$ is a~standard DG bilinear form; its consistency implies that the EG bilinear form is also consistent since $V^k_{\ell,m} \subset \IP_k(\T_h)$.
\section{Stability analysis}\label{SEC:stability}
The stability of the method can be obtained exactly as the stability of the DG methods. Thus,
\begin{theorem}
 The EG solution is $L^\infty(0,T;L^2(\Omega))$ stable.
\end{theorem}
\begin{proof}
 We test $a_h$ with $\varphi = U$, use the identity $\tfrac{1}{2} (g')^2 = g g'$ and integrate by parts to obtain
 \begin{align*}
  \frac{1}{2} \partial_t \| U & \|^2_{L^2(\Omega)} +\,\frac{1}{2} \int_\Omega (\nabla \cdot \vec a) U^2  \dx
\underbrace{-~\frac{1}{2}  \int_\Sigma \jump{ U^2 } \cdot \vec a \ds + \int_{\Sigma \setminus \Gamma_-} \!\!\!\!\!\!\! \upw{U} \vec a \cdot\jump{U} \ds }_{ = \frac{1}{2} \int_\Sigma | \vec a \cdot \Nu_\Omega | \jump{U}^2 \ds }\\
  ~=~ & \int_\Omega f U \dx + \int_{\Gamma_\textup D} u_\text D | \vec a \cdot \Nu_\Omega | U \ds + \int_{\Gamma_\textup F} g_\textup F U \ds\\
  ~\le~ & \int_\Omega f U \dx + \frac{1}{2} \int_{\Gamma_\text D} \left[ |\vec a \cdot \Nu_\Omega | u_\text D^2 + |\vec a \cdot \Nu_\Omega | \jump{U}^2 \right] \ds 
  +\frac{1}{2} \int_{\Gamma_\textup F} \left[ \frac{g^2_\textup F}{|\vec a \cdot \Nu_\Omega|} + |\vec a \cdot \Nu_\Omega | \jump{U}^2 \right] \ds, 
 \end{align*}
 where the last inequality follows from the Young's and Cauchy--Schwarz inequalities and uses the assumption $\vec a(t, \vec{x}) \cdot \Nu_\Omega \le -\delta$ on $\Gamma_\textup F$. This directly implies the $L^\infty(L^2)$-stability without exponential growth of constants if $f \equiv 0$ and $\nabla \cdot \vec a \ge 0$ after integrating with respect to time. Otherwise Gr\"onwall's, Young's, and Cauchy--Schwarz inequalities give the result (after moving $\tfrac{1}{2} \int_\Omega (\nabla \cdot \vec a) U^2  \dx$ to the right-hand side).
\end{proof}
\section{Error analysis}\label{SEC:error}
For the error analysis, we need some auxiliary results:
\begin{lemma}
 The operator $\bpi$ of \eqref{bpi} is an~orthogonal projection into $\IP_\ell(\T_\H) + \IP_m(\T_{\H|h})$ with respect to the $L^2$-inner product, i.e.,
 \begin{equation}\label{orth}
  \int_\Omega (\bpi u - u)\varphi \dx = 0 \qquad \forall \varphi \in \IP_\ell(\T_\H) + \IP_m(\T_{\H|h}).
 \end{equation}
\end{lemma}
\begin{proof}
 Result follows directly from the fact that $I^k_\H u - u \in L^2(\Omega)$ and the $L^2$-orthogonality of $\orthProj$.
\end{proof}
\begin{lemma}[Best approximation property of $\orthProj$]
 For all $\varphi \in \IP_l(\T_{\H}) + \IP_m(\T_{\H|h})$, $g \in L^2(\Omega)$, and all $\elem \in \T_{\H}$,
 \begin{equation}\label{bestappr}
  \|\orthProj g - g\|_{L^2(\elem)} \leq \|\varphi - g\|_{L^2(\elem)}
 \end{equation}
\end{lemma}
\begin{proof}
 Follows directly from the $L^2$-orthogonality of $\Pi^{\ell,m}_{\H|h}$ and the possibility to localize the projection to all $\elem \in \T_\H$.
\end{proof}
\begin{lemma}[Inverse inequality]
 Let $({\T_\H})_{\H \in \mathcal{I}}$ be a regular  mesh sequence. There exists a constant $C>0$ such that for all $\H \in \mathcal{I}$, all $\varphi \in \IP_k(\T_\H)$, and all $\elem \in \T_\H$
 \begin{equation}\label{inv}
  \vert \varphi \vert_{H^1(\elem)} \leq C~ h_\elem^{-1} \|\varphi\|_{L^2(\elem)}.
 \end{equation}
\end{lemma}
\begin{proof}
 This is \cite[Lem.~1.44]{PietroErn}.
\end{proof}
\begin{lemma}[Discrete trace inequality]
 Let $({\T_\H})_{\H \in \mathcal{I}}$ be a regular mesh sequence. There exists a constant $C>0$ such that for all $\H \in \mathcal{I}$, all $v \in \IP_k(\T_\H)$, all $\elem \in \T_\H$, and all $F \in \F$ with $F \subset \partial \elem$ 
 \begin{equation}\label{tr}
  \| \varphi \|_{L^2(F)} \leq C~ h^{-1/2}_\elem \|\varphi\|_{L^2(\elem)}.
 \end{equation}
\end{lemma}
\begin{proof}
 This is \cite[Lem.~1.46]{PietroErn}.
\end{proof}
\begin{lemma}[Continuous trace inequality]
 Let $({\T_\H})_{\H \in \mathcal{I}}$ be a regular mesh sequence. There exists a constant $C>0$ such that for all $\H \in \mathcal{I}$, all $\elem \in \T_\H$, all $v \in H^1(\elem)$, and all $F\in\F$ with $F \subset \partial \elem$
 \begin{equation}\label{ctr}
  \|v\|_{L^2(F)}^2 \leq C~\left(\vert v\vert_{H^1(\elem)} + h_\elem^{-1}\|v\|_{L^2(\elem)}\right)\|v\|_{L^2(\elem)}.
 \end{equation}
\end{lemma}
\begin{proof}
 This is \cite[Lem.~1.49]{PietroErn}.
\end{proof}
\begin{lemma}[Approximation property]
 Let $({\T_\H})_{\H \in \mathcal{I}}$ be a regular mesh sequence. There exists a constant $C > 0$ such that for all $v \in H^{k+1}(\Omega)$, all $\H \in \mathcal I$, and all $\elem \in \T_\H$
 \begin{align}
  |\Pi^k_\H v - v|_{H^m(\elem)} &\leq C~h_\elem^{k+1-m} |v|_{H^{k+1}(\elem)}\quad \text{for }k\in \IN \cup \{0\},\label{apprPi}\\
  |I^k_\H v - v|_{H^m(\elem)} &\leq C~h_\elem^{k+1-m} |v|_{H^{k+1}(\elem)}\quad \text{for }k \in \IN \setminus \{0\}.\label{apprInt}
 \end{align}
\end{lemma}
\begin{proof}
 The first inequality is \cite[Theo.~3.29]{KnabnerAngermann}. The second inequality is \cite[Lem.~1.58]{PietroErn}.
\end{proof}
\begin{lemma}
 Let $({\T_\H})_{\H \in \mathcal{I}}$ be a regular mesh sequence. There exists a constant $C>0$ such that for all $v \in W^{1,\infty}(\Omega)$, all $\H \in \mathcal I$, and all $\elem \in \T_\H$
 \begin{align}
  |\Pi^0_\H v\|_{L^\infty(\elem)} &\leq  \|v\|_{L^\infty(\elem)},\label{pi0stab}\\
  \|\Pi^0_\H v - v\|_{L^\infty(\elem)} &\leq h_\elem |v|_{W^{1,\infty}(\elem)}.\label{pi0conv}
 \end{align}
\end{lemma}
\begin{proof}
 The first inequality is the observation that $\Pi^0_\H v$ is the element-wise mean of $v$ which needs to be smaller than or equal to its essential maximum. The second inequality is a simple combination of  \cite[Theo.~3.24 \& 3.26]{KnabnerAngermann}.
\end{proof}

Next, we formulate our main result.
\begin{theorem}\label{TH:convergence}
 Let $(\T_{\H|h})$ be a weakly quasi-uniform mesh (double) sequence, and let $u \in H^1(0,T;H^{k+1}(\Omega))$, $k \in \IN \setminus \{0\}$. Then, the EG approximation $U$ converges in $L^\infty(L^2)$ to the analytical solution $u$, i.e., there exists $C$ independent of $\H$ and $h$ such that
 \begin{equation*}
  \| u - U \|_{L^\infty(0,T;L^2(\Omega))}^2 \le C \sum_{\elem \in \T_\H} h_{\H|\elem}^{2\tilde m} \H_{\elem}^{2(k - m)} | u |^2_{H^1(0,T;H^{k+1}(\elem))}
 \end{equation*}
 with $\tilde m = m + 1/2$ for simplicial meshes and $\ell \ge k - 1$ or general meshes and $\ell = k$ (i.e. in the case of DG). Otherwise, $\tilde m = m$.
\end{theorem}
\begin{proof}
Defining
\begin{equation*}
 e_u \coloneqq U - \bpi u \quad \text{ and } \quad \theta_u \coloneqq \bpi u - u,
\end{equation*}
we have due to the consistency and since $e_u \in V^k_{\ell,m}$ that
\begin{equation*}
 a(e_u + \theta_u, e_u) = 0 \quad \text{ for almost every } t \in (0,T).
\end{equation*}
This can be rewritten as
\begin{align*}
 \frac{1}{2} \partial_t  \| e_u \|^2_{L^2(\Omega)} +& \frac{1}{2} \int_\Sigma | \vec a \cdot \Nu | \jump{e_u}^2 \ds 
 =  \underbrace{ - \frac{1}{2} \int_\Omega (\nabla \cdot \vec a) e_u^2 \dx }_{ \eqqcolon \Xi_1 } - \underbrace{ \int_{\Sigma \setminus \Gamma_-} \!\!\!\!\!\! \upw{\theta_u} \vec a \cdot\jump{e_u} \ds }_{ \eqqcolon \Xi_2 } - \underbrace{ \int_\Omega \partial_t \theta_u e_u \dx }_{ \eqqcolon \Xi_3 }\\
 & + \underbrace{ \sum_{\elem \in \T_{\H|h}} \int_\elem \theta_u (\vec a - \Pi_{\H|h}^0 \vec a) \cdot \nabla e_u \dx }_{ \eqqcolon \Xi_4 } + \underbrace{ \sum_{\elem \in \T_{\H|h}} \int_\elem \theta_u \Pi_{\H|h}^0 \vec a \cdot \nabla e_u \dx }_{ \eqqcolon \Xi_5 }.
\end{align*}
We can immediately deduce that:
\begin{itemize}
 \item If $\nabla \cdot \vec a \ge 0$ then $\Xi_1 \le 0$ holds, and this term can be moved to the left hand side and integrated into the energy norm. This is consistent with the continuous case, when the mass sinks lead to an increase in the stability.
 \item If $\vec a$ is element-wise constant, then $\Xi_4 \equiv 0$.
 \item If the mesh is simplicial, and the globally continuous polynomials are from the space $\mathcal P_k$, then $\Pi_{\H|h}^0 \vec a \cdot \nabla e_u \in \IP_{k-1}(\T_{\H}) + \IP_{m-1}(\T_{\H|h})$. Using \eqref{orth} yields $\Xi_5 \equiv 0$, provided that $\ell \ge k-1$.
 \item If the mesh is quadrilateral, and the globally continuous polynomials are from the space $\mathcal Q_k$, then $\Pi_{\H|h}^0 \vec a \cdot \nabla e_u \in \IP_{k}(\T_{\H}) + \IP_{m}(\T_{\H|h})$. Using \eqref{orth} yields $\Xi_5 \equiv 0$ provided that $\ell = k$, i.e., in the case of DG.
\end{itemize}
Next, we estimate terms $\Xi_{\cdot}$:
\begin{align*}
 |\Xi_1| & ~\le~ \frac{1}{2}~\| \nabla \cdot \vec a \|_{L^\infty(\Omega)} \| e_u \|^2_{L^2(\Omega)},\\
 |\Xi_2| & ~\le~ \int_{\Sigma \setminus \Gamma_-} \upw{ \theta_u }^2 \ds + \frac{1}{4}  \int_\Sigma | \vec a \cdot \Nu | \jump{e_u}^2 \ds\\
 & ~\le~ C~ \sum_{\elem \in \T_{\H|h}} \int_{\partial \elem} \theta_u^2 \ds + \frac{1}{4} \int_\Sigma | \vec a \cdot \Nu | \jump{e_u}^2 \ds\\
 & ~\overset{\eqref{ctr}}{\le}~ C \sum_{\elem \in \T_{\H|h}} \left(\vert \theta_u\vert_{H^1(\elem)} + h_\elem^{-1} \|\theta_u\|_{L^2(\elem)}\right) \| \theta_u \|_{L^2(\elem)}
   + \frac{1}{4} \int_\Sigma | \vec a \cdot \Nu | \jump{e_u}^2 \ds\\
 & ~\le~C ~\sum_{\elem \in \T_{\H|h}}  h_\elem \vert\theta_u\vert_{H^1(\elem)}^2  + ~C \sum_{\elem \in \T_{\H|h}} h_\elem^{-1}\|\theta_u\|_{L^2(\elem)}^2
  + \frac{1}{4} \int_\Sigma | \vec a \cdot \Nu | \jump{e_u}^2 \ds,\\
 |\Xi_3| & ~\le~ \frac{1}{2}~\| \partial_t \theta_u \|_{L^2(\Omega)} + \frac{1}{2}~\| e_u \|_{L^2(\Omega)},\\
 |\Xi_4| & ~\le~ \sum_{\elem \in \T_{\H|h}} \| \theta_u \|_{L^2(\elem)} \| \vec a - \Pi_{\H|h}^0 \vec a \|_{L^\infty(\elem)} \| \nabla e_u \|_{L^2(\elem)}\\
 & ~\overset{\eqref{tr},\eqref{pi0conv}}{\le}~ C~ \sum_{\elem \in \T_{\H|h}} \| \theta_u \|_{L^2(\elem)} h_\elem | \vec a |_{W^{1,\infty}(\elem)} h^{-1}_\elem \| e_u \|_{L^2(\elem)}\\
 & ~\le~ | \vec a |^2_{W^{1,\infty}(\Omega)} \sum_{\elem \in \T_{\H|h}} \| \theta_u \|^2_{L^2(\elem)} + C \sum_{\elem \in \T_{\H|h}} \| e_u \|^2_{L^2(\elem)},\\
 |\Xi_5| & ~\le~ \sum_{\elem \in \T_{\H|h}} \| \theta_u \|_{L^2(\elem)} \| \Pi_{\H|h}^0 \vec a \|_{L^\infty(\elem)} \| \nabla e_u \|_{L^2(\elem)}\\
 & ~\overset{\eqref{tr},\eqref{pi0stab}}{\le}~ C~ \sum_{\elem \in \T_{\H|h}} \| \theta_u \|_{L^2(\elem)} \| \vec a \|_{L^\infty(\elem)} h^{-1}_\elem \| e_u \|_{L^2(\elem)}\\
 & ~\le~ \| \vec a \|^2_{L^\infty(\Omega)} \sum_{\elem \in \T_{\H|h}} h^{-2}_\elem \| \theta_u \|^2_{L^2(\elem)} + C \sum_{\elem \in \T_{\H|h}} \| e_u \|^2_{L^2(\elem)}.
\end{align*}
This would give the desired result (after applying Gr\"onwall's inequality -- if needed) provided that we could find good estimates for the terms involving $\theta_u$, and $\vert \theta_u \vert_{H^1(\elem)}$. Note that only the norm $\| \nabla \cdot \vec a\|_{L^\infty(\Omega)}$ enters the exponential term in the Gr\"onwall estimate.

We consider the cases $\elem \in \mathcal S_h$ and $\elem \in \mathcal S_\H$ separately. In the first case, we can estimate
\begin{align*}
 &\sum^{\hat\elem \in \T_\H \setminus \mathcal S_\H} _{\mathcal S_h \ni \elem \subset \hat \elem} h^r_\elem~\| \bpi u - u \|^2_{L^2(\elem)} \le \sum_{\hat\elem \in \T_\H \setminus \mathcal S_\H} h^r_{\H|\hat \elem} \| \bpi u - u  \|^2_{L^2(\hat\elem)}\\
 &\qquad  =  \sum_{\hat\elem \in \T_\H \setminus \mathcal S_\H} h^r_{\H|\hat \elem} \| \orthProj (u - I^k_\H u) - (u - I^k_\H u) \|^2_{L^2(\hat \elem)}\\
 &\qquad  \overset{\eqref{bestappr}}{\leq} \sum_{\hat\elem \in \T_\H \setminus \mathcal S_\H} \ h^r_{\H|\hat \elem}\| \Pi^m_{\H|h} (u - I^k_\H u) - (u - I^k_\H u) \|^2_{L^2(\hat \elem)}\\
 &\qquad  = \sum^{\hat\elem \in \T_\H \setminus \mathcal S_\H}_{\mathcal S_h \ni \elem \subset \hat \elem} h^r_{\H|\hat \elem} \| \Pi^m_{\H|h} (u - I^k_\H u) - (u - I^k_\H u) \|^2_{L^2(\elem)}\\
 &\qquad  \overset{\eqref{apprPi}}{\leq}~C~ \sum_{\hat\elem \in \T_\H \setminus \mathcal S_\H} h_{\H|\hat \elem}^{2m + 2 + r}~ | u - I^k_\H u |^2_{H^{m+1}(\hat \elem)}\\
 &\qquad \overset{\eqref{apprInt}}{\leq}~ C~\sum_{\hat\elem \in \T_\H \setminus \mathcal S_\H} h_{\H|\hat \elem}^{2m + 2 + r}  {\H}_{\hat \elem}^{2k-2m}~ \vert u \vert_{H^{k+1}(\hat \elem)}^2.
\end{align*}

In the second case, we obtain for $\elem \in \mathcal S_\H$ using the same arguments
\begin{equation*}
 \sum_{\elem \in \mathcal S_\H} \H^r_\elem \| \bpi u - u \|^2_{L^2(\elem)} \le~C~\sum_{\elem \in \mathcal S_\H} \H^{2k + 2 + r}_\elem |u|^2_{H^{k+1}(\elem)}.
\end{equation*}

The estimate for $\vert \theta_u \vert_{H^1(\elem)}$ is conducted analogously. Here, the projection $\bpi^\star u \coloneqq I^k_\H u + \Pi^m_{\H|h} (u - I^k_\H u)$ is used to obtain
\begin{align*}
 \sum^{\hat\elem \in \T_\H \setminus \mathcal S_\H} _{\mathcal S_h \ni \elem \subset \hat \elem} h^r_\elem \vert \bpi u - u \vert^2_{H^1(\elem)}
 \overset{\eqref{inv}}{\leq}~C~\sum^{\hat\elem \in \T_\H \setminus \mathcal S_\H}_{\mathcal S_h \ni \elem \subset \hat \elem} \left(h_\elem^r\vert  \bpi^\star u - u \vert_{H^1(\elem)}^2 +h_\elem^{r-2} \|\bpi u - \bpi^\star u\|_{L^2(\elem)}^2  \right),
\end{align*}
which gives the needed estimate after inserting $\pm u$ into the second summand and redoing the aforementioned arguments. Collecting all terms gives the result.
\end{proof}
\begin{remark}
 This result is not optimal, since it uses high regularity of the temporal derivative. However, in the case of DG, i.e. $m = -1$ and $\ell = k$, the proof can be streamlined by replacing $\bpi$ (and $\bpi^\star$) by $\Pi^k_{\H}$---this also implies that the initial data is constructed using an~orthogonal projection with respect to the $L^2$-norm. Here, also the distinction between simplices and quadrilaterals/hexahedrons becomes unnecessary, and the polynomial approximation spaces may all be of $\mathcal P_k$ type. This results in $\Xi_3 = \Xi_5 = 0$ and yields the optimal estimate
 \begin{align*}
  \| u - U \|_{L^\infty(0,T;L^2(\Omega))}^2 ~\le~  C \sum_{\elem \in \T_\H} h_{\H|\elem}^{2m+1} \H_{\elem}^{2(k - m)} | u |^2_{L^2(0,T;H^{k+1}(\elem))}
  ~\le~  C \H^{2k+1} | u |^2_{L^2(0,T;H^{k+1}(\Omega))},
 \end{align*}
 where $u$ is only assumed to be an element of $L^2(0,T;H^{k+1}(\Omega))$.
\end{remark}

\section{Numerical results}\label{SEC:numerics}
\subsection{Analytical convergence test}

%
In order to verify the convergence of the numerical schemes, we use the method of manufactured solution.  On the domain $\Omega \coloneqq (0,1)\times (0,1)$ and the time interval $J \coloneqq(0,1/2)$, we define the analytical solution $u(t, x_1, x_2)$ and velocity filed $\vec a(t, x_1, x_2)$ by
\begin{align*}
u(t, x_1, x_2) = \cos(7 x_1) \cos(7 x_2) + \exp(-t), \qquad
\vec a(t,x_1,x_2) =
\begin{pmatrix}
\exp(x_1/2 + x_2/2)\\
\exp(x_1/2 - x_2/2)
\end{pmatrix}.
\end{align*}
The right-hand side $f$ of the problem is chosen so that $u$ and $\vec a$ satisfy \eqref{EQ:advection}. We prescribe Dirichlet boundary conditions on the inflow boundary, i.e., $\Gamma_\mathrm{D} \coloneqq \Gamma_-$, and use $u_\mathrm{D}\coloneqq u_{|_{(x_1, x_2)\in \Gamma_D}}$ and $u_0 \coloneqq u_{|_{t=0}}$.


Let $r$ and $R$ denote the refinement levels for the meshes with element sizes $h$ and $\mathcal H$, respectively.
The initial mesh ($R$=$r$=1) consisting of four triangles is obtained by diagonally subdividing $\Omega$; finer meshes are produced by connecting the edge midpoints of every triangle. As temporal discretization, we use an explicit SSP Runge--Kutta method with $s = k+1$ stages.

We utilize the EG method with polynomial orders $k$ and $\ell$ on the coarse grid (of refinement level $R$) enriched by the DG method of order at most $m$ on the fine grid (of refinement level $r$). 


Our implementation currently supports the approximation orders up to two. This yields four possible combinations of $k,\ell$, and $m$. In Table~\ref{TAB:100}, the $(r,R)$-th entry corresponds to the $L^2$-error at time $t = 1/2$ using the refinement levels $r$ and $R$. 



\begin{table}[ht]
\setlength{\tabcolsep}{3.5pt}
 \begin{tabular}{c|ccccc|ccccc}
  \toprule
space & \multicolumn{5}{c}{$k = 1$, $\ell=0$, $m=0$} & \multicolumn{5}{c}{$k=2$, $\ell=0$, $m=0$}  \\
  \midrule
  \diagbox{$r$}{$R$}
    & 1 & 2 & 3 & 4 & 5& 1 & 2 & 3 & 4 & 5\\
\midrule
  1 & 5.85E-01&	--- &	--- &	--- & ---
& 4.48E-01&	---&	---&	--- & ---\\
  2 & 5.42E-01&	2.85E-01&	--- &	--- & ---
& 2.95E-01&	7.50E-02&	---&	---& ---\\
  3 & 2.95E-01&	1.95E-01&	7.24E-02&	---& ---
& 2.21E-01&	7.39E-02&	1.06E-02&	---& ---\\
  4 & 1.60E-01&	1.10E-01&	7.37E-02&	1.88E-02&  ---
& 1.23E-01&	6.79E-02&	9.67E-03&	1.46E-03& ---\\
  5 & 8.40E-02&	5.78E-02&	4.18E-02&	1.99E-02&   4.80E-03
& 6.46E-02&	3.93E-02&	8.73E-03&	1.27E-03& 2.07E-04\\
  6 & 4.33E-02&	2.98E-02&	2.21E-02&	1.12E-02&   5.16E-03
& 3.33E-02&	2.10E-02&	5.03E-03&	1.17E-03& 1.59E-04\\
  7 & 2.21E-02&	1.53E-02&	1.15E-02&	5.90E-03&   2.89E-03
& 1.70E-02&	1.10E-02&	2.68E-03&	6.72E-04& 1.49E-04\\
  \midrule
space & \multicolumn{5}{c}{$k=2$, $\ell=1$, $m=0$} & \multicolumn{5}{c}{$k=2$, $\ell=1$, $m=1$}  \\
  \midrule
  \diagbox{$r$}{$R$}
    & 1 & 2 & 3 & 4 & 5& 1 & 2 & 3 & 4 & 5\\
\midrule
1 & 4.47E-01&	---&	---&	---& ---
& 4.47E-01&	---&	---&	---& ---\\
2 & 2.95E-01&	7.11E-02&	---&	---& ---
& 1.72E-01&	7.11E-02&	---&	---& ---\\
3 & 2.21E-01&	7.34E-02&	9.80E-03&	---& ---
& 6.62E-02&	5.96E-02&	9.80E-03&	---& ---\\
4 & 1.23E-01&	6.76E-02&	9.59E-03&	1.29E-03& ---
& 1.76E-02&	1.63E-02&	7.57E-03&	1.29E-03& ---\\
5 & 6.46E-02&	3.91E-02&	8.69E-03&	1.26E-03& 1.63E-04
& 4.51E-03&	4.24E-03&	2.10E-03&	1.00E-03& 1.63E-04\\
6 & 3.33E-02&	2.09E-02&	5.00E-03&	1.17E-03& 1.59E-04
& 1.15E-03&	1.09E-03&	5.41E-04&	2.79E-04& 1.27E-04\\
7 & 1.70E-02&	1.09E-02&	2.66E-03&	6.70E-04& 1.49E-04
& 3.29E-04&	2.77E-04&	1.37E-04&	7.18E-05& 3.56E-05\\
  \bottomrule
 \end{tabular}
 \caption{Analytical convergence test: $L^2$-errors for $V^k_{\ell,m}$, $0\leq m \leq \ell < k \leq 2$.}\label{TAB:100}
\end{table}

We observe that our analytical convergence rates are confirmed by the numerical tests; however, somewhat better convergence (by an order of ca. $1/2$) is apparent. This is a~well-known phenomenon also experienced in numerical experiments for the DG method on regular meshes. Moreover, the first subcell refinement step has the tendency to show a deteriorated rate of convergence -- presumably due to an~increasing constant when switching between the two branches in the proof of Theorem~\ref{TH:convergence} (discriminating between locally refined and not locally refined elements).

In Fig.~\ref{fig:convPlot}, one can see the respective convergence plots for different local refinement strategies. Note that the solutions for $V^2_{0,0}$ and $V^2_{1,0}$ are very similar, their error plots in Fig.~\ref{fig:convPlot} lie on top of each other. 
The error plots for the local refinements with $h = \H/4$ (dashed lines) and $h = 2\H^2$ (solid lines) are shown in Fig.~\ref{fig:convPlot} (left). We observe that the slopes of the error plots match (or exceed by ca. $1/2$) the convergence rates in Theorem \ref{TH:convergence}. 
In Fig.~\ref{fig:convPlot} (right), the convergence for fixed $\H = 1/2$ and successively refined $h$ is shown.  In line with Theorem \ref{TH:convergence}, we observe order of convergence one for numerical methods with $m = 0$ and order of convergence two for~$V^2_{1,1}$.
\begin{figure}[ht]
	\includegraphics[draft = false, width = \linewidth]{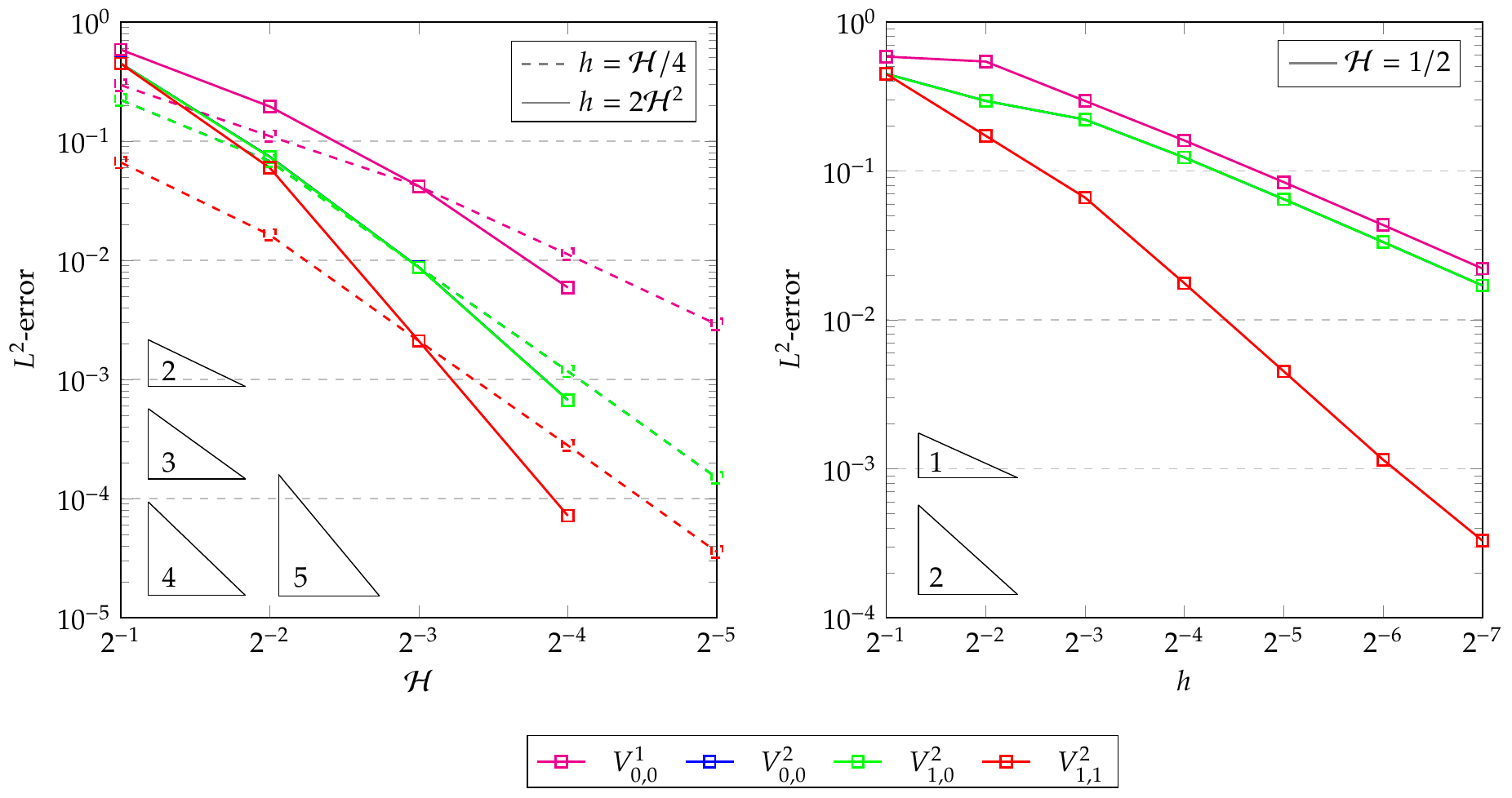}
	\caption{Analytical convergence test: Log plot of $L^2$-errors for $V^k_{\ell,m}$, $0\leq m \leq \ell < k \leq 2$ and different refinement strategies. }
	\label{fig:convPlot}
\end{figure}

\subsection{Solid body rotation}
As the next benchmark problem, we use solid body rotation test proposed by LeVeque~\cite{LeVeque1996}. It consists of a slotted cylinder, a sharp cone, and a smooth hump (see Fig.~\ref{fig:solidbodycoarsemesh} (right)) that are placed in a square domain $\Omega = (0, 1)\times (0,1)$ and transported by a time-independent velocity field 
\begin{equation*}
\vec a(t,x_1, x_2) =
\begin{pmatrix}
0.5- x_2\\x_1 - 0.5
\end{pmatrix}
\end{equation*}
in a counterclockwise rotation about $J = (0,2\pi)$. Using $r = 0.0225$ and $G(\vec x, \vec x_0) = \frac{1}{0.15}\|\vec x - \vec x_0\|_2$, we choose the following initial data
\small
\begin{equation*}
u_0(\vec{x}) = \left\{
\begin{array}{ll}
  1                                               & \text{if}\quad (x_1 - 0.5)^2 + (x_2 - 0.75)^2 \le r\;\land\;(x_1\le0.475 \lor x_1\ge0.525 \lor x_2\ge0.85)\\
  1-G\left(\vec{x},{0.5 \choose 0.25}\right)                       & \text{if}\quad (x_1 - 0.5)^2 + (x_2 - 0.25)^2 \le r\\[0.5ex]
  \frac{1}{4}\left(1+\cos\left(\pi G\left(\vec{x},{0.25 \choose 0.5}\right)\right)\right) & \text{if}\quad (x_1 - 0.25)^2 + (x_2 - 0.5)^2 \le r\\
  0                                               & \text{otherwise} 
\end{array}
\right.
\end{equation*}
\normalsize
At the inlet  $\Gamma_-$, we prescribe the Dirichlet boundary condition $u = 0$. The right-hand side of the advection equation is given by $f = 0$. In order to obtain a~discrete initial condition preserving the bounds of the analytical solution ($0 \le U_{|_{t=0}} \le 1$), we define $u_0$ using the $L^2$-projection into the space of piecewise constant functions instead of our special EG projection operator $\Pi^{\ell,m}_{\H|h}$.
\begin{figure}
\includegraphics[draft = false, height=0.35\textheight]{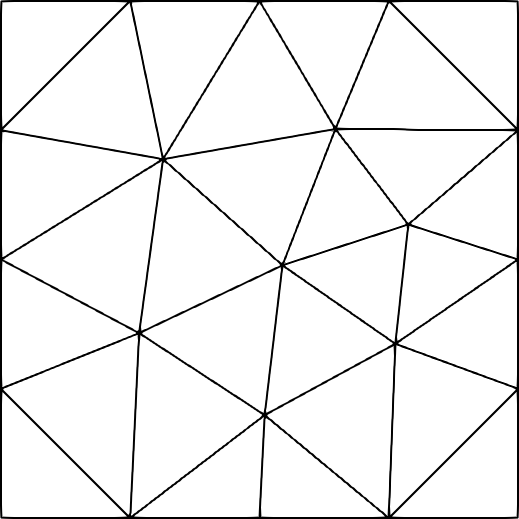}
\hspace{0.1\linewidth}
\includegraphics[draft = false, height=0.35\textheight, trim=120 20 110 60,clip]{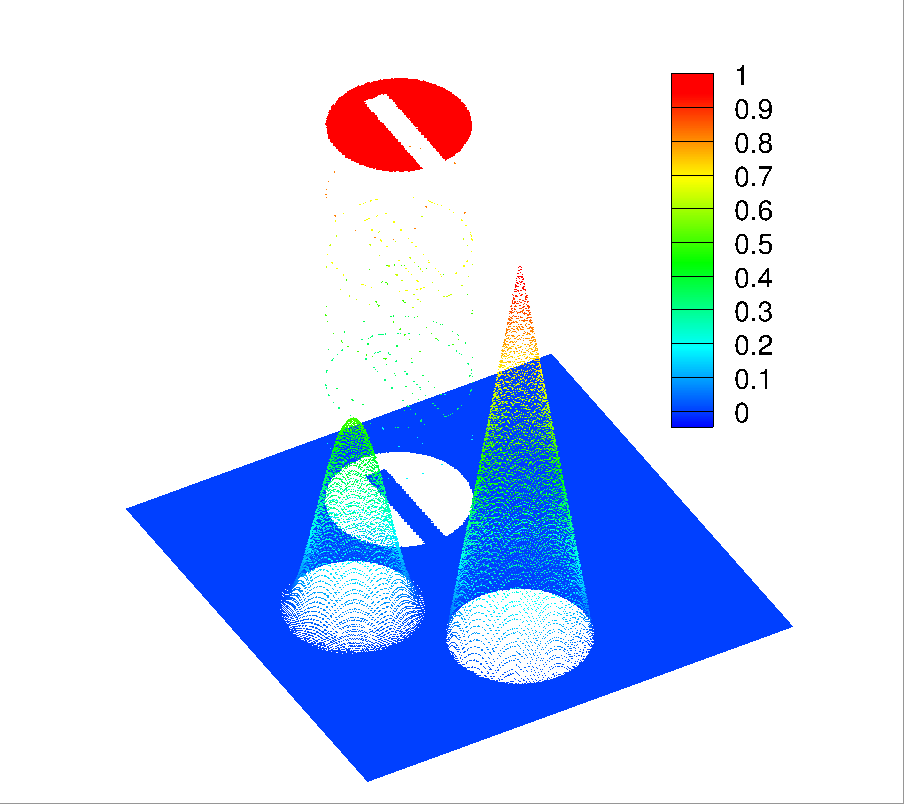}
\caption{Solid body rotation: Coarse mesh (left); initial condition projected on the space of element-wise constant polynomials on mesh with $R = 7$ (right).}
	\label{fig:solidbodycoarsemesh}
\end{figure}

\begin{figure}
	\begin{tabular}{cc}
		a) $R = 4, r = 4$,\;  $L^2$-err: 1.28E-01 & b) $R = 4, r = 5$,\;  $L^2$-err: 1.15E-01 \\
		\includegraphics[draft  = false, width=0.33\linewidth, trim=100 20 100 50,clip]{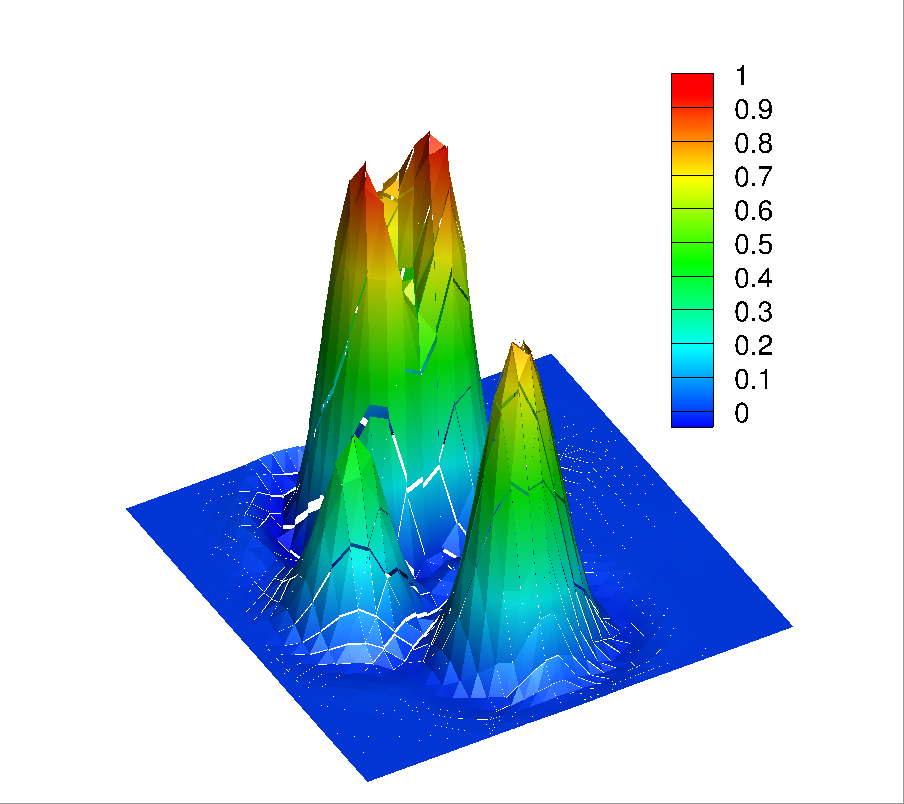} & 
		\includegraphics[draft = false, width=0.33\linewidth, trim=100 20 100 50,clip]{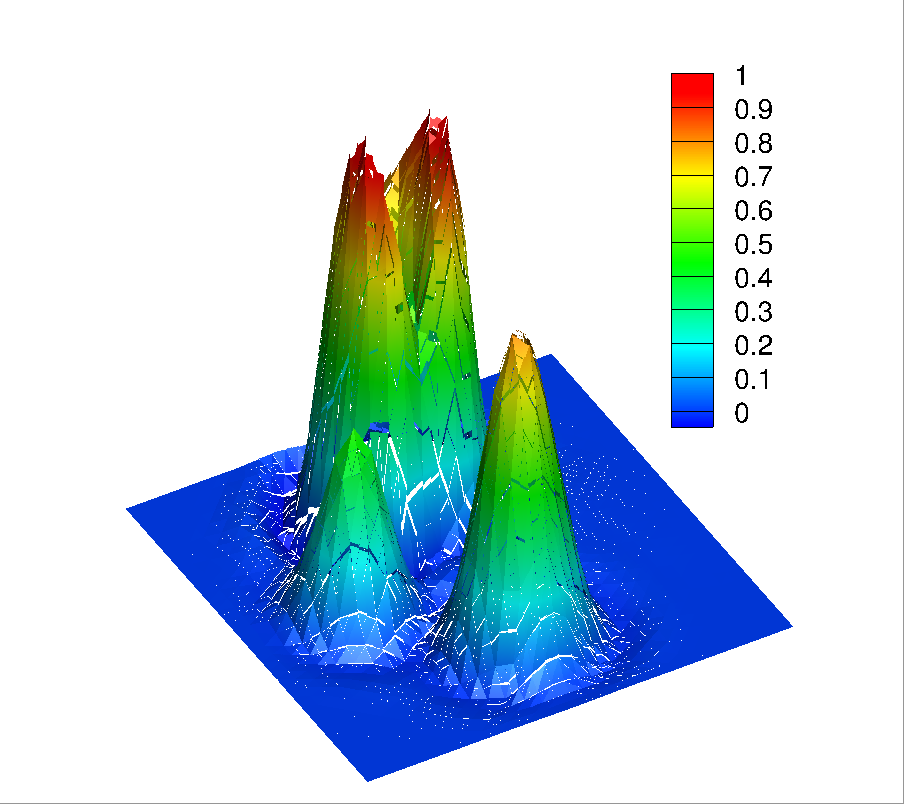} \\[1ex]
		c) $R = 4, r = 6$, \; $L^2$-err: 1.03E-01 & d) $R = 4, r = 7$,\;  $L^2$-err: 9.23E-02 \\
		\includegraphics[draft = false, width=0.33\linewidth, trim=100 20 100 50,clip]{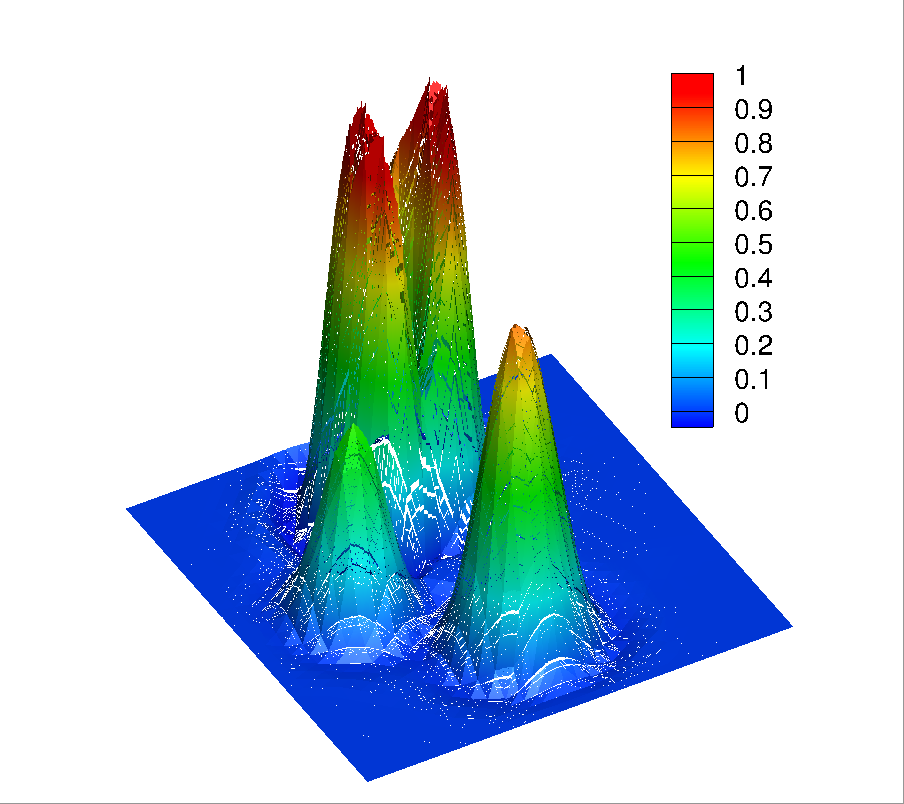} & 
		\includegraphics[draft = false, width=0.33\linewidth, trim=100 20 100 50,clip]{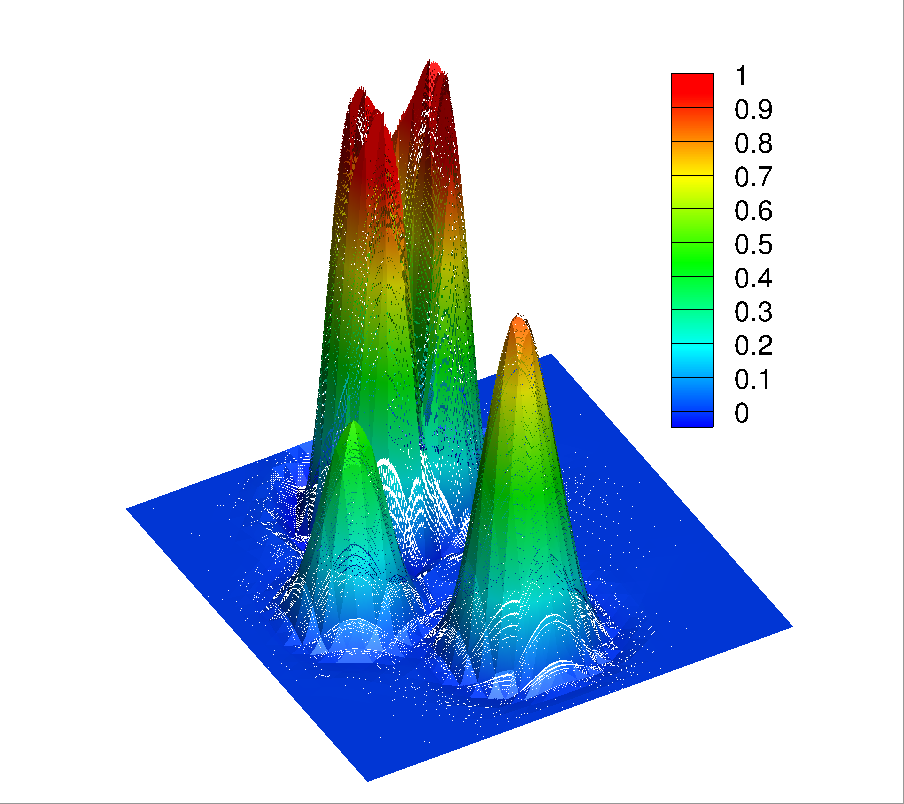}\\[1ex]
		e) $R = 5, r = 5$,\;  $L^2$-err: 8.80E-02 & f) $R = 6, r = 6$,\;  $L^2$-err: 6.49E-02  \\
		\includegraphics[draft = false, width=0.33\linewidth, trim=100 20 100 50,clip]{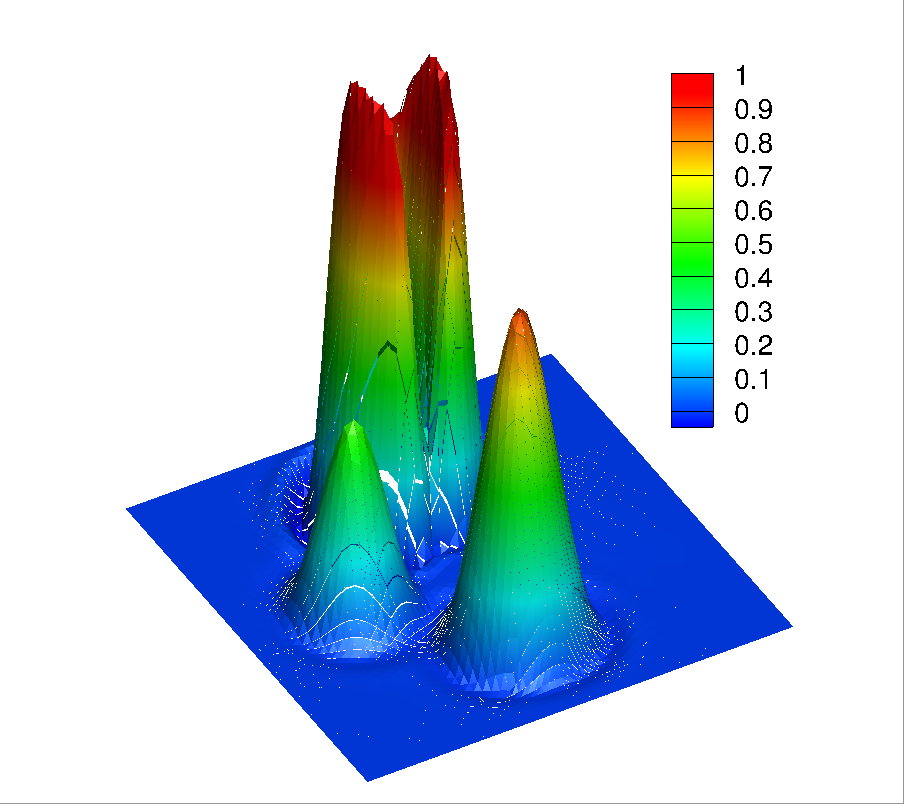} &
		\includegraphics[draft = false, width=0.33\linewidth, trim=100 20 100 50,clip]{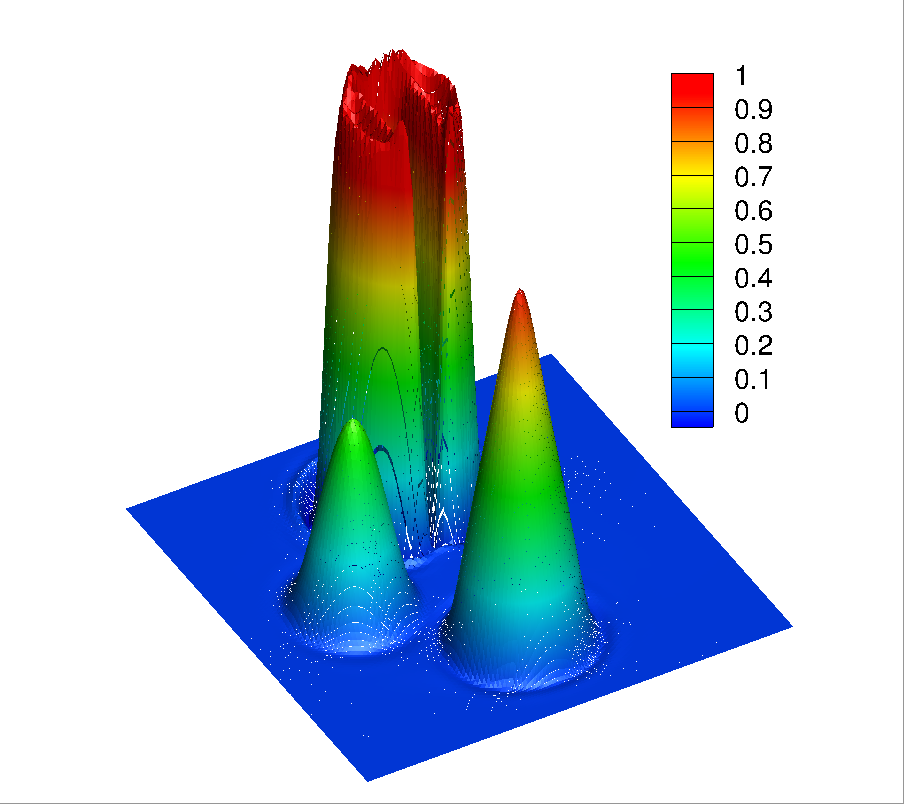}
	\end{tabular}
\caption{Solid body rotation: Final state after one rotation.\label{fig:solidbody}}	
\end{figure}

\begin{figure}
\includegraphics[draft  = false, width=0.48\linewidth, trim=20 30 20 80,clip]{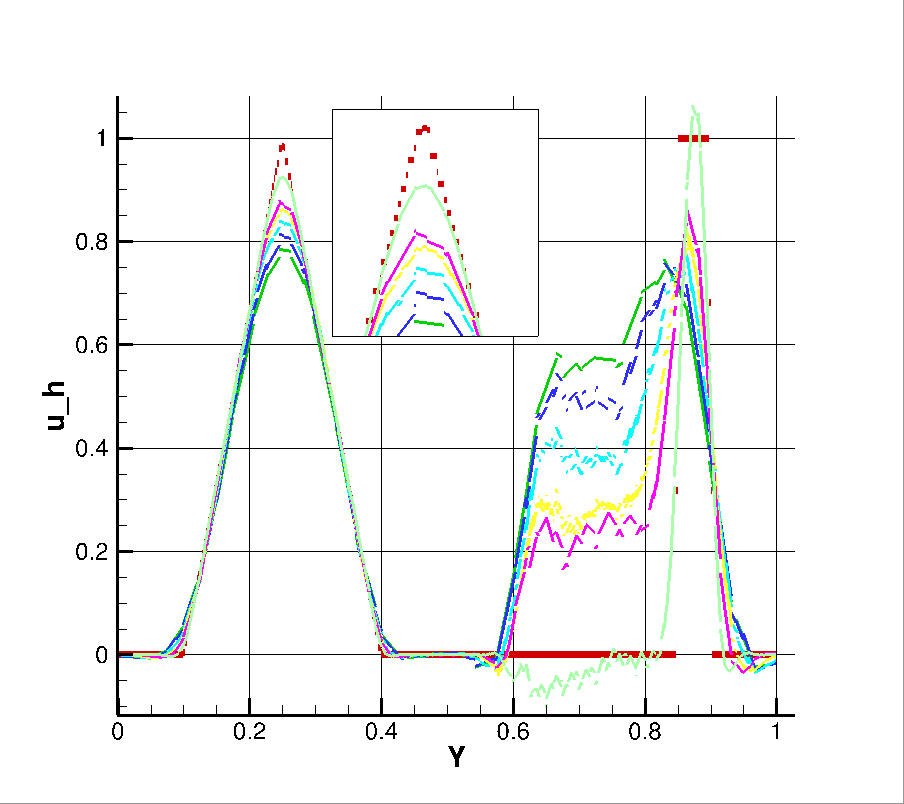}
\hfill
\includegraphics[draft  = false, width=0.48\linewidth, trim=20 30 20 80,clip]{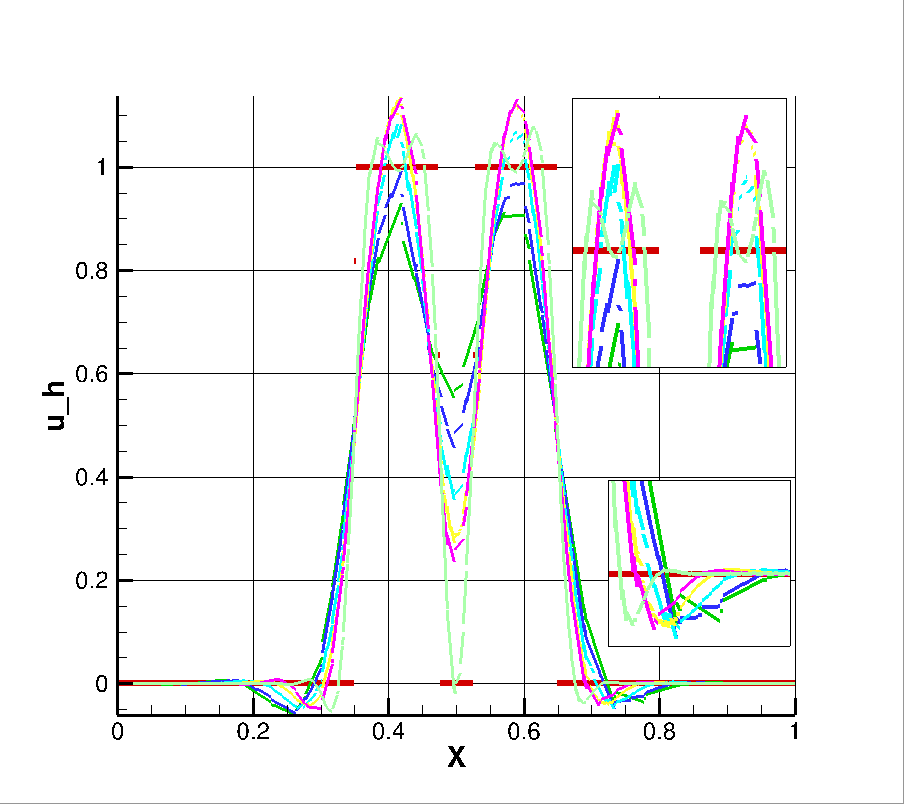}\\
\includegraphics[draft  = false, width=0.165\linewidth, trim=170 300 380 460,clip]{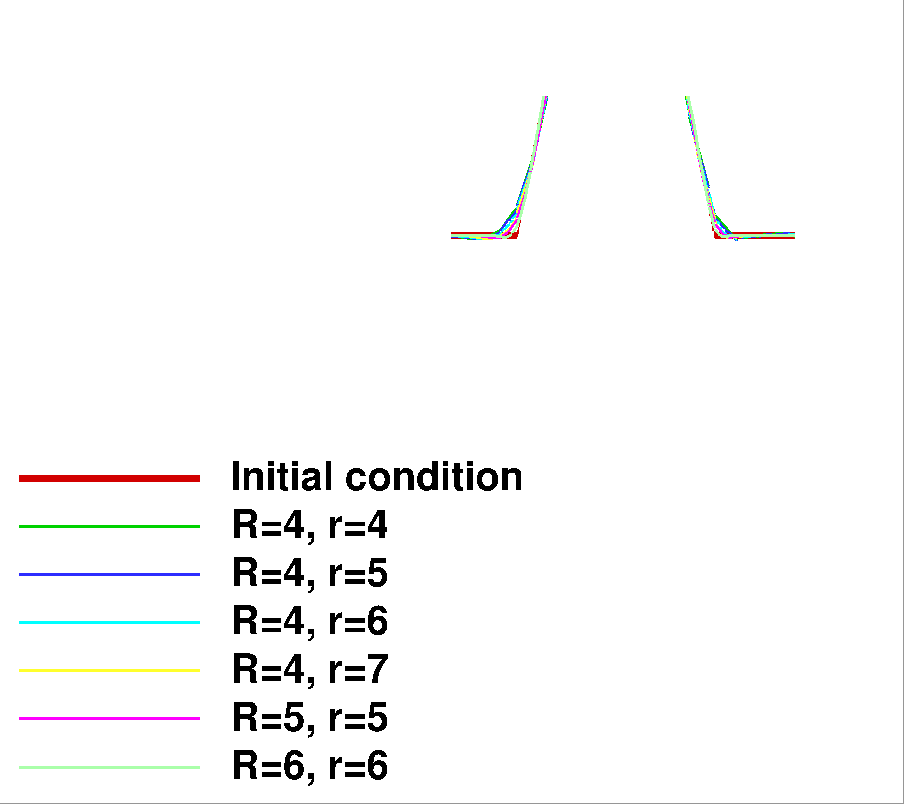}
\includegraphics[draft  = false, width=0.165\linewidth, trim=170 250 380 510,clip]{pictures/solid_body_cross_legend}\hspace{-7mm}
\includegraphics[draft  = false, width=0.165\linewidth, trim=170 203 380 557,clip]{pictures/solid_body_cross_legend}\hspace{-7mm}
\includegraphics[draft  = false, width=0.165\linewidth, trim=170 156 380 604,clip]{pictures/solid_body_cross_legend}\hspace{-7mm}
\includegraphics[draft  = false, width=0.165\linewidth, trim=170 109 380 651,clip]{pictures/solid_body_cross_legend}\hspace{-7mm}
\includegraphics[draft  = false, width=0.165\linewidth, trim=170 62 380 698,clip]{pictures/solid_body_cross_legend}\hspace{-7mm}
\includegraphics[draft  = false, width=0.165\linewidth, trim=170 15 380 745,clip]{pictures/solid_body_cross_legend}
\caption{Solid body rotation: Cross-sections at positions $x$=0.5 (top), $y$=0.75 (bottom).\label{fig:solidbody-cross}}	
\end{figure}

The results presented in Figs.~\ref{fig:solidbody} and \ref{fig:solidbody-cross} illustrate the stabilizing effect of piecewise-constant ($k=1,\ l=m=0$) subcell enrichments on different mesh levels. The standard CG approximation would produce spurious oscillations in the whole domain. The EG method localizes them to a small neighborhood of the slotted cylinder, while producing well-resolved approximations of the smooth hump and sharp cone.

\section{Conclusions}\label{SEC:conclusions}
In this article, we introduced and investigated a~generalization of the enriched Galerkin method that relies on a~two-mesh enrichment with discontinuous functions of arbitrary order. The method was shown to be stable and to converge at the same rate as the discontinuous Galerkin method. Our numerical results demonstrated good agreement with the \textit{a~priori} convergence analysis, although the experimental rates of convergence on regular meshes exceeded those of the analysis by approximately $1/2$ -- in line with the well-known results for the DG method. Our investigation suggests that using local subcell enrichment is an~exceptionally flexible discretization approach for representing solutions of locally highly varying regularity without incurring too much computational overhead. While global subcell enrichments do not offer the same savings in the number of degrees of freedom as the classical EG method, local enrichment in selected cells is ideally suited for $hp$-adaptivity purposes.

In the future work, we plan to extend this methodology to more complicated applications (e.g. shallow--water equations) and look into the possibility of using subcell enrichments in $hp$-adaptive bound-preserving finite element schemes.

\bibliographystyle{amsplain}
\bibliography{EG_Hh}
\end{document}